\documentclass{amsart}

\usepackage[utf8]{inputenc}
\usepackage{color}
\usepackage{xcolor}
\usepackage{hyperref}
\usepackage{amssymb}
\usepackage{float}
\usepackage{tikz,tikz-cd}
\usepackage{forest}
\usepackage{nicefrac}
\usepackage{enumerate}
\usepackage[nameinlink, capitalize, noabbrev]{cleveref}
\usetikzlibrary{shapes.geometric,decorations.markings}
\usetikzlibrary{decorations.pathreplacing}
\usetikzlibrary{fit}
\usetikzlibrary{automata}
\usetikzlibrary{positioning}
\usetikzlibrary{intersections}
\tikzset{mytext/.style={font=\small, text=black}}

\hypersetup{
    colorlinks=true,
    linkcolor=teal,
    citecolor=magenta,
    }

\newtheorem{Theorem}{Theorem}

\newtheorem{Conjecture}{Conjecture}
\newtheorem{Corollary}[Conjecture]{Corollary}
\newtheorem{proposition}{Proposition}[section]
\newtheorem{lemma}[proposition]{Lemma}
\newtheorem{corollary}[proposition]{Corollary}
\newtheorem{theorem}[proposition]{Theorem}
\newtheorem{fact}[proposition]{Fact}
\newtheorem{Question}[Conjecture]{Question}
\theoremstyle{definition}

\newtheorem{definition}[proposition]{Definition}

\numberwithin{equation}{section}

\title{Cyclicity, hypercyclicity and randomness in self-similar groups}
\author{Jorge Fariña-Asategui}
\address{Jorge Fariña-Asategui: Centre for Mathematical Sciences, Lund University, 223 62 Lund, Sweden -- Department of Mathematics, University of the Basque Country UPV/EHU, 48080 Bilbao, Spain}
\email{jorge.farina\_asategui@math.lu.se}
\thanks{The author is supported by the Spanish Government, grant PID2020-117281GB-I00, partly with FEDER funds. The author also acknowledges support from the Walter Gyllenberg Foundation from the Royal Physiographic Society of Lund}

\begin{document}

\begin{abstract}
We introduce the concept of cyclicity and hypercyclicity in self-similar groups as an analogue of cyclic and hypercyclic vectors for an ope-rator on a Banach space. We derive a sufficient condition for cyclicity of non-finitary automorphisms in contracting discrete automata groups. In the profinite setting we prove that fractal profinite groups may be regarded as measure-preserving dynamical systems and derive a sufficient condition for the ergodicity and the mixing properties of these dynamical systems. Furthermore, we show that a Haar-random element in a super strongly fractal profinite group is hypercyclic almost surely as an application of Birkhoff's ergodic theorem for free semigroup actions.
\end{abstract}

\maketitle

\section{introduction}
\label{section: introduction}

Self-similar groups are strongly connected to dynamical systems. A great example of this connection are the iterated monodromy groups of complex polynomials, which were used by Bartholdi and Nekrashevych to answer a well-known open problem in complex dynamics \cite{Thurston}. 

Here our aim is to establish a connection between self-similar profinite groups and measure-preserving dynamical systems. We start by regarding a profinite group~$G$ as a probability space via its unique Haar measure $\mu_G$. Assume further that $G$ is a self-similar subgroup of $\mathrm{Aut}~T$ for some regular rooted tree $T$. Then we may regard~$T$ as a free monoid and define an action $\mathcal{T}$ of $T$ on $(G,\mu_G)$ by considering sections, i.e. each $v\in T$ yields an operator $\mathcal{T}_v$ given by $\mathcal{T}_v(g):=g|_v$ for every $g\in G$; see \cref{section: Preliminaries} for the unexplained terms here and elsewhere in the introduction. A natural question here is whether this action $\mathcal{T}$ of $T$ on $(G,\mu_G)$ is measure-preserving. If so, then the tuple $(G,\mu_G,\mathcal{T},T)$ becomes a measure-preserving dynamical system and we are further interested on the ergodicity and the mixing properties of this dynamical system. We obtain the following:

\begin{Theorem}
\label{Theorem: fractal and ergodic}
    Let $G\le \mathrm{Aut}~T$ be a self-similar profinite group:
    \begin{enumerate}[\normalfont(i)]
        \item if $G$ is fractal then $(G,\mu_G,\mathcal{T},T)$ is a measure-preserving dynamical system;
        \item if $G$ is super strongly fractal then $(G,\mu_G,\mathcal{T},T)$ is ergodic and strongly mixing.
    \end{enumerate}
\end{Theorem}

In a general dynamical system describing the orbits of different points in the space is often of special interest. These orbits are closely related to the invariant subsets for a measure-preserving transformation. In the case of a single operator acting on a Banach space these invariant subspaces have been thoroughly studied and continue to be an active area of research in functional analysis, see \cite{Book, Enflo1, Lomonosov, Read}. General invariant subspaces are often studied via the so-called \textit{cyclic} and \textit{hypercyclic} vectors, i.e. those vectors which are not contained in a proper invariant subspace or subset respectively; see \cite{Cyclicity}. We define the analogue of cyclic and hypercyclic vectors for self-similar groups in terms of the natural action via sections of the free monoid~$T$ on $(G,\mu_G)$; see \cref{section: example} for precise definitions. The following questions arise naturally:

\begin{Question}
\label{Question: fractal contains cyclic}
    Does every fractal profinite group contain a cyclic automorphism? And a hypercyclic automorphism?
\end{Question}

By \cref{Theorem: fractal and ergodic} a super strongly fractal profinite group may be regarded as an ergodic measure-preserving dynamical system and therefore we may apply the machinery in ergodic theory to study Haar-random elements. Indeed, we shall make use of a generalized version of Birkhoff's theorem for free semigroup actions to prove the following:

\begin{Theorem}
\label{Theorem: super strongly fractal cyclic}
    Let $G$ be a super strongly fractal profinite group. Then a Haar-random element of~$G$ is hypercyclic with probability 1.
\end{Theorem}

In particular \cref{Theorem: super strongly fractal cyclic} implies that every super strongly fractal profinite group contains a hypercyclic and therefore a cyclic automorphism, answering \cref{Question: fractal contains cyclic} in the case of super strongly fractal groups.

\begin{Corollary}
    Every super strongly fractal profinite group contains a hypercyclic automorphism.
\end{Corollary}

Another consequence of \cref{Theorem: super strongly fractal cyclic} is that Haar-random elements may not be used in the study of the self-similar Hausdorff spectrum of super strongly fractal profinite groups; compare \cite{AbertVirag, JorgeSpectra}. In fact, Abért and Virág proved in \cite[Theorem~7.2]{AbertVirag} that three random elements in the pro-$p$ Sylow subgroup $W_p$ of the automorphism group of the $p$-adic tree generate topologically a closed subgroup of Hausdorff dimension 1 in $W_p$. Thus by a recent result of the author \cite[Theorem~1.2]{JorgeSpectra}, 3-random generated subgroups of $W_p$ cannot be self-similar. \cref{Theorem: super strongly fractal cyclic} generalizes this result, showing that no topologically finitely generated random closed subgroup of $W_p$ can be contained in a proper self-similar closed subgroup of~$W_p$.

Lastly, \cref{Theorem: fractal and ergodic} also opens the door to applying methods from dynamical systems in the study of self-similar groups. The author together with Radi have recently utilized \cref{Theorem: fractal and ergodic} in \cite{JorgeSanti} to obtain  new results on the fixed-point proportion of super strongly fractal groups, which itself has a wide range of applications in arithmetic dynamics; compare \cite{FPP1,FFP3,FFP2,FFP4}.

\subsection*{\textit{\textmd{Organization}}} In \cref{section: Preliminaries} we introduce the background material needed in the subsequent sections. Cyclicity and hypercyclicity in self-similar groups are introduced in \cref{section: example} via a motivating example which is both an $\mathbb{F}_p$-vector space and a self-similar group. We further study cyclic automorphisms in discrete automata groups to conclude \cref{section: example}. Finally in \cref{section: fractal groups} we prove \textcolor{teal}{Theorems} \ref{Theorem: fractal and ergodic} and \ref{Theorem: super strongly fractal cyclic} showing an application of Birkhoff's ergodic theorem for free semigroup actions to the theory of self-similar profinite groups.

\subsection*{\textit{\textmd{Notation}}} An event in a probability space will be said to happen \textit{almost surely} if it happens with probability 1. We denote the finite field of $p$ elements by $\mathbb{F}_p$. Self-similar groups will be assumed to act on the tree on the right so composition will be written from left to right. We shall use exponential notation for group actions on the tree. Finally, we denote by $\# S$ the cardinality of a finite set $S$.

\subsection*{Acknowledgements} I would like to thank my advisors Gustavo A. Fernández-Alcober and Anitha Thillaisundaram for their valuable feedback which contributed greatly to the final version of this paper. Finally, I would like to thank the anonymous referee for their valuable comments.

\subsection*{Statements and declarations} There are no competing interests.

\section{Preliminaries}
\label{section: Preliminaries}

\subsection{The Haar measure on profinite groups}

Let $G$ be a profinite group. We define the \textit{Borel algebra} of $G$ as the $\sigma$-algebra generated by all the open subsets of~$G$. The elements of the Borel algebra of $G$ will be called \textit{Borel sets}. 

A \textit{$\pi$-system} on $G$ is a non-empty collection $\mathcal{A}$ of subsets of $G$ closed under finite intersections, i.e. if $A,B\in \mathcal{A}$ then $A\cap B\in \mathcal{A}$. In particular, a family of neighborhood bases for each point in $G$ forms a $\pi$-system. The $\sigma$-algebra generated by such a $\pi$-system is precisely the Borel algebra of $G$. We recall the following standard fact about $\pi$-systems:

\begin{fact}
\label{fact: pi systems}
    Let $\mu_1$ and $\mu_2$ be two probability measures on a set $X$ and let $P$ be a $\pi$-system on $X$. If for every $A\in P$ we have $\mu_1(A)=\mu_2(A)$, then $\mu_1(S)=\mu_2(S)$ for any set $S$ in the $\sigma$-algebra generated by the $\pi$-system $P$.
\end{fact}

A \textit{Haar measure} on $G$ is a measure $\mu$ on the Borel algebra of $G$ satisfying the following:
\begin{enumerate}[\normalfont(i)]
    \item it is finite and normalized, i.e. $\mu(G)=1$;
    \item it is left-translation invariant, i.e. $\mu(gS)=\mu(S)$ for every $g\in G$ and every Borel subset $S\subseteq G$;
    \item it is outer regular on Borel sets, i.e.
    $$\mu(S)=\inf\{\mu(U)~:~ S\subseteq U \text{ and }U\text{ open}\};$$
    \item it is inner regular on open sets, i.e.
    $$\mu(U)=\sup\{\mu(K)~:~ K\subseteq U \text{ and }K\text{ compact}\}.$$
\end{enumerate}

By Haar's theorem a profinite group $G$ admits a unique Haar measure, which we shall denote by $\mu_G$. From \cref{fact: pi systems} and both the left-translation invariance and the uniqueness of the Haar measure the following is immediate:

\begin{corollary}
\label{corollary: equality of measures}
    Let $\mu$ be a measure on the Borel algebra in $G$. If $\mu$ coincides with the Haar measure $\mu_G$ on a neighborhood basis of the identity in $G$, then $\mu=\mu_G$.
\end{corollary}

Note that for an open subgroup $H\le G$ we have $\mu_G(H)=|G:H|^{-1}$.

\subsection{Measure-preserving dynamical systems}

Let $(X,\mu)$ be a probability space and let $A$ be a monoid. We fix an action $\mathcal{T}$ of $A$ on $(X,\mu)$ and we say this monoid action is \textit{measure-preserving} if for any measurable subset $S\subseteq X$ and any $a\in A$ we have $\mu(\mathcal{T}_a^{-1}(S))=\mu(S)$, where $\mathcal{T}_a$ is the operator associated to the action of $a$. Then the tuple $(X,\mu,\mathcal{T},A)$ is called a \textit{measure-preserving dynamical system}.

A measurable subset $S\subseteq X$ such that $\mathcal{T}_a^{-1}(S)= S$ for every $a\in A$ is called \textit{$\mathcal{T}$-invariant}. We say that a measure-preserving dynamical system $(X,\mu,\mathcal{T},A)$ is \textit{ergodic} if for every $\mathcal{T}$-invariant measurable subset $S\subseteq X$ we have $\mu(S)\in \{0,1\}$. 

A stronger notion in ergodic theory is that of strong mixing. We adapt the usual definition for a single operator to the action of a free monoid $F$. We define a dynamical system $(X,\mu,\mathcal{T},F)$ to be \textit{strongly mixing} if for any pair $S_1,S_2$ of measurable subsets and any sequence $\{a_k\}_{k\ge 1}\subseteq F\setminus\{1\}$ the following limit exists:
$$\lim_{k\to\infty}\mu(\mathcal{T}_{a_1\dotsb a_k}^{-1}(S_1)\cap S_2)=\mu(S_1)\cdot \mu(S_2).$$
For a $\mathcal{T}$-invariant measurable subset $S$ the strong mixing condition implies that $\mu(S)=\mu(S)^2$, i.e. $\mu(S)\in \{0,1\}$. Thus strong mixing implies ergodicity:

\begin{proposition}
\label{proposition: strongly mixing implies ergodic}
    If a measure-preserving dynamical system $(X,\mu,\mathcal{T},F)$ is strongly mixing then it is ergodic.
\end{proposition}

\subsection{Self-similar groups}

For a natural number $m\ge 1$ and a finite set of $m$ symbols $\{1,\dotsc,m\}$, we define the \textit{free monoid on $m$ generators} as the monoid consisting of finite words with letters in $\{1,\dotsc,m\}$. The free monoid can be identified with the \textit{$m$-adic tree}, i.e. the rooted tree $T$ where each vertex has exactly $m$ immediate descendants. The words in~$T$ of length exactly $n$ form the \textit{$n$th level of} $T$. We may also use the term level to refer to the number $n$. For a vertex $v\in T$ we write $|v|$ for the length of $v$ as a word with letters in $\{1,\dotsc,m\}$.

Let $\mathrm{Aut~}T$ be the group of graph automorphisms of the $m$-adic tree $T$, i.e. those bijective maps on the set of vertices of $T$ preserving adjacency. Such automorphisms fix the root and permute the vertices at the same level of $T$. For any $1\le n\le \infty$, the \textit{$n$th truncated tree} $T_n$ consists of the vertices at distance at most $n$ from the root. Note that $T_\infty=T$. We denote the group of automorphisms of the $n$th truncated tree by $\mathrm{Aut~}T_n$.

Let $g\in\mathrm{Aut~}T$ and $v\in T$. For $1\le n\le \infty$ we define the \textit{section of $g$ at $v$ of depth }$n$ as the unique automorphism $g|_v^n\in\mathrm{Aut~}T_n$ such that $(vu)^g=v^gu^{g|_v^n}$ for every $u$ of length at most $n$. For $n=\infty$ we write  $g|_v$ for simplicity and call it the \textit{section of $g$ at $v$}, while for $n=1$ we shall call $g|_v^1$ the \textit{label of $g$ at $v$}. For every pair $g,h\in \mathrm{Aut}~T$, $v\in T$ and $n\ge 1$ we have the equality
$$(gh)|_v^n=(g|_v^n)(h|_{v^g}^n).$$

For every $n\ge 1$ the subset $\mathrm{St}(n)$ of automorphisms fixing all the vertices of the $n$th level of $T$ is a normal subgroup of finite index in $\mathrm{Aut~}T$, the \textit{$n$th level stabilizer}. Similarly, for any vertex $v\in T$ we define its \textit{vertex stabilizer} $\mathrm{st}(v)$ as the subgroup of automorphisms fixing the vertex $v$. The group $\mathrm{Aut~}T$ is a countably based profinite group with respect to the topology induced by its action on each level of the tree. We call this topology the \textit{congruence topology}. For each vertex~$v$ we shall define the continuous homomorphisms $\varphi_v:\mathrm{st}(v)\to \mathrm{Aut~}T$ given by $g\mapsto g|_v$ and, more generally, for every $1\le n\le \infty$ we define $\varphi_v^n:\mathrm{st}(v)\to \mathrm{Aut}~T_n$ given by $g\mapsto g|_v^n$. We further define the continuous isomorphism $\psi:\mathrm{Aut}~T\to \big(\mathrm{Aut~}T\times\overset{m}{\ldots}\times \mathrm{Aut~}T\big)\rtimes \mathrm{Sym}(m)$ given by $g\mapsto (g|_{1},\dotsc,g|_m)g|_\emptyset^1$. 

We define \textit{finitary automorphisms of depth $n$} as those automorphisms of the $m$-adic tree with trivial labels at the $k$th level of the tree for every $k\ge n$. Finitary automorphisms of depth 1 are called \textit{rooted automorphisms} and they may be identified with elements in the symmetric group $\mathrm{Sym}(m)$. More generally, the group of finitary automorphisms of depth $n$ is isomorphic to the group $\mathrm{Aut}~T_n$ of automorphisms of the $n$th truncated tree via the projection of the action of $\mathrm{Aut}~T$ to the $n$th truncated tree $T_n$.

For a subgroup $G\le \mathrm{Aut}~T$ we define $\mathrm{st}_G(v):=\mathrm{st}(v)\cap G$ and $\mathrm{St}_G(n):=\mathrm{St}(n)\cap G$ for any vertex $v$ and any level $n\ge 1$, respectively. The quotients $G_n:=G/\mathrm{St}_G(n)$ are called the \textit{congruence quotients} of $G$ and their elements the $n$-\textit{patterns} of $G$.

Let $G$ be a subgroup of $\mathrm{Aut}~T$. We say that $G$ is \textit{self-similar} if for any $g\in G$ and any vertex $v\in T$ we have $g|_{v}\in G$. We shall say that $G$ is \textit{fractal} if $G$ is self-similar and $\varphi_v(\mathrm{st}_G(v))=G$ for every $v\in T$. A stronger version of fractality is that of \textit{super strongly fractal} groups, where $\varphi_v(\mathrm{St}_G(n))=G$ for every vertex $v$ at the $n$th level of $T$ for every level $n\ge 1$. We shall not assume level-transitivity neither for fractal nor for super strongly fractal groups.

\section{Cyclic and hypercyclic automorphisms}
\label{section: example}

In this section we motivate the definition of cyclic and hypercyclic automorphisms via an example. We consider an elementary $p$-abelian self-similar profinite group $A$. We shall use the fact that elementary $p$-abelian groups may be regarded as $\mathbb{F}_p$-vector spaces to identify $A$ with the $\mathbb{F}_p$-vector space $V$ consisting of infinite sequences over $\mathbb{F}_p$. Then we show that self-similar closed subgroups of $A$ correspond precisely to the $\mathbb{F}_p$-vector subspaces of $V$ invariant under the left-shift operator. Therefore we study the cyclic and hypercyclic vectors of $V$ for the left-shift operator, which motivates the analogous concepts of cyclic and hypercyclic automorphisms in self-similar profinite groups. Furthermore, we study some discrete subspaces of $V$ and motivate the discrete analogue of cyclic automorphisms. Lastly, we give sufficient conditions for discrete automata groups to admit such cyclic automorphisms and apply this criterion to some of the main examples in the literature.

\subsection{The group $A$ and the space $V$}

Let us fix $\sigma:=(1\,\dotsb \, p)\in \mathrm{Sym}(p)$. We define the elementary $p$-abelian pro-$p$ group $A$ acting on the $p$-adic tree $T$  as 
$$A:=\overline{\langle g_n~:~ n\ge 0\rangle}\le \mathrm{Aut}~T,$$
where, for each $n\ge 0$, the automorphism $g_n$ is given by
    $$g_n|^1_v:=\begin{cases}
        \sigma,&\text{ if }v\text{ is at level }k= n;\\
        1,&\text{ if }v\text{ is at level }k\ne n.
    \end{cases}$$

Any element $g\in A$ admits a unique factorization of the form
    $$g=\prod_{n\ge 0}g_n^{e_n},$$
    where $0\le e_n\le p-1$. Note that the group $A$ is simply the profinite group $G_\mathcal{S}\le W_p$ defined in \cite[Section 4]{JorgeSpectra} given by the defining sequence $\mathcal{S}=\{S_n\}_{n\ge 0}$, where $S_n:=D_{p^n}(\sigma)$.

Furthermore, since $A$ is elementary $p$-abelian, it admits an $\mathbb{F}_p$-vector space structure. We define the $\mathbb{F}_p$-vector space $V$ consisting of infinite sequences over $\mathbb{F}_p$ with pointwise addition and scalar multiplication. Note that the vectors in $V$ may be simply regarded as functions $\mathbb{N}\cup\{0\}\to \mathbb{F}_p$ and $V$ is a compact topological space with respect to the product topology. 

There is an $\mathbb{F}_p$-linear homeomorphism $f:A\to V$ given by
    $$f(g)=f\big(\prod_{n\ge 0}g_n^{e_n}\big):=(e_n)_{n\ge 0}.$$

Let us define the \textit{left-shift} operator $\tau:V\to V$ via
$$\tau(e_n)_{n\ge 0}:=(e_{n+1})_{n\ge 0}.$$

We note that considering sections of an element in $A$ at the first level of $T$ coincides with the action of the left-shift operator on $V$ via the $\mathbb{F}_p$-homeomorphism~$f$, as shown in the proof of the following lemma:

\begin{lemma}
\label{lemma: self-similar invariant for A}
    Let $H\le A$ be a (closed) subgroup. The (closed) subgroup $H$ is self-similar if and only if $f(H)$ is a $\tau$-invariant (closed) subspace.
\end{lemma}
\begin{proof}
    Note that a group $G$ is self-similar if and only if for every $g\in G$ and every~$v$ at the first level of the tree $g|_v\in G$. Now note that for $g=\prod_{n\ge 0}g_n^{e_n}\in A$ and for any vertex $v$ at the first level of $T$ we have
    $$g|_v=\prod_{n\ge 0}g_n^{e_{n+1}}.$$
    Therefore 
    $$f(g|_v)=(e_{n+1})_{n\ge 0}=\tau((e_n)_{n\ge 0})=\tau(f(g)),$$
    i.e. the following diagram commutes
$$
    \begin{tikzcd}
    A \arrow{r}{f} \arrow[swap]{d}{\cdot |_v} & V \arrow{d}{\tau} \\
    A \arrow{r}{f} & V
\end{tikzcd}
$$
and the equivalence in the statement is now apparent. The same holds replacing ``subgroup$"$ with ``closed subgroup$"$ as $f$ is a homeomorphism.
\end{proof}

\subsection{Cyclicity and hypercyclicity in self-similar profinite groups}

A vector $v\in V$ is said to be \textit{cyclic} for the left-shift operator $\tau$ if the closed subspace generated by the $\tau$-orbit of $v$ is $V$. Furthermore if the $\tau$-orbit of $v$ itself is dense in $V$ then~$v$ is called \textit{hypercyclic}. In other words, cyclic (hypercyclic) vectors are those non-trivial vectors which are not contained in any proper $\tau$-invariant closed subspace (subset) of $V$. The commutative diagram above suggests that we may use the homeomorphism $f:A\to V$ to obtain analogous definitions for self-similar profinite groups.

Let $G\le \mathrm{Aut}~T$ be a self-similar group, where $T$ is a regular rooted tree. Given a subgroup $H\le G$, we define the \textit{self-similar closure} of $H$ in $G$, denoted by $H^{\mathrm{SS}}$, as the smallest self-similar subgroup of $G$ containing $H$. In other words
$$H^{\mathrm{SS}}:=\bigcap K,$$
where $K$ ranges over all self-similar subgroups of $G$ such that $H\le K\le G$. An explicit construction of the self-similar closure is given by the following straightforward lemma:

\begin{lemma}
    For $H\le G$ we have $H^{\mathrm{SS}}=\langle h|_v~:~ h\in H\text{ and }v\in T\rangle$.
\end{lemma}

\begin{definition}[Cyclicity and hypercyclicity in self-similar profinite groups]
\label{def: cyclic profinite}
    Let $G\le \mathrm{Aut}~T$ be a self-similar closed subgroup. We say $g\in G$ is \textit{cyclic} in~$G$ if $\overline{\langle g\rangle^{\mathrm{SS}}}=G$. The automorphism $g$ is called \textit{hypercyclic} if $\overline{\{g|_v~:~ v\in T\}}=G$. 
\end{definition}

In other words: an automorphism $g\in G$ is cyclic if it is not contained in any non-trivial proper self-similar closed subgroup of $G$ and hypercyclic if it is not contained in a non-trivial proper self-similar closed subset of $G$.

Cyclic and hypercyclic automorphisms can be used to deduce properties of the subgroup structure of a self-similar closed subgroup:

\begin{lemma}
\label{lemma: cyclic vectors in profinite groups}
    A self-similar closed subgroup $G\le \mathrm{Aut}~T$ contains a non-trivial proper self-similar closed subgroup (subset) if and only if $G$ contains a non-trivial non-cyclic (non-hypercyclic) automorphism.
\end{lemma}
\begin{proof}
    If $G$ contains a non-trivial proper self-similar closed subgroup $H$ then for any $1\ne h\in H$ we have $1\ne \overline{\langle h\rangle^{\mathrm{SS}}}\le H<G$, so $h$ is a non-trivial non-cyclic automorphism in $G$. For the converse, if $1\ne h\in G$ is non-cyclic by definition $\overline{\langle h\rangle^{\mathrm{SS}}}$ yields a non-trivial proper self-similar closed subgroup of $G$. For self-similar closed subsets the proof is analogous.
\end{proof}

A self-similar closed subgroup is defined to be \textit{cyclically (hypercyclically) generated as a profinite group} if it contains a cyclic (hypercyclic) automorphism in the sense of \cref{def: cyclic profinite}. \cref{Theorem: super strongly fractal cyclic} shows that actually every super strongly fractal profinite group is hypercyclically generated, but it does not yield explicit generators as its proof is probabilistic. Let us show an explicit example.

Let $H\le \mathrm{Sym}(m)$ be a transitive subgroup. We define the \textit{iterated wreath product} $W_H\le \mathrm{Aut}~T$ as 
$$W_H:=\{g\in \mathrm{Aut}~T~:~ g|_v^1\in H\text{ for every }v\in T\}.$$
The set $S:=\bigcup_{n\ge 0} P_n$, where each $P_n$ is the set of $n$-patterns of $W_H$, is countable, so let us write $\{s_i\}_{i\ge 1}$ for the countable set of finitary automorphisms corresponding to each pattern in $S$. Let us consider the automorphism
 $$g:=\prod_{i\ge 1} s_i*2\overset{i}{\dotsb} 21\in W_H,$$
where $g*v$ stands for the unique automorphism in $\mathrm{St}_G(|v|)$ such that $(g*v)|_v=g$ and $(g*v)|_w=1$ for every other vertex $w$ at the same level as $v$. Then by construction $g|_{2\overset{i}{\dotsb} 21}=s_i$ for every $i\ge 1$, as for any $1\le i<j$ the vertices $2\overset{i}{\dotsb}21$ and $2\overset{j}{\dotsb}21$ do not lie in a common geodesic from the root. Therefore
\begin{align*}
        \overline{\{g|_v~:~ v\in T\}}=\overline{\{s_i~:~i\ge 1\}}=&W_H
\end{align*}
yielding that $g$ is hypercyclic.

\subsection{Cyclicity in self-similar discrete groups}

\begin{definition}[Cyclicity in discrete groups]
\label{def: cyclicity discrete}
    Let $G\le \mathrm{Aut}~T$ be a self-similar discrete subgroup. We say $g\in G$ is \textit{cyclic} in~$G$ if $\langle g\rangle^{\mathrm{SS}}=G$.
\end{definition}

 In other words, $g\in G$ is cyclic if there is no proper self-similar subgroup of $G$ containing $g$.
The following is proved exactly as \cref{lemma: cyclic vectors in profinite groups}:

\begin{lemma}
    A self-similar group $G\le \mathrm{Aut}~T$ contains a non-trivial proper self-similar subgroup if and only if $G$ contains a non-trivial non-cyclic automorphism.
\end{lemma}

We say that a self-similar discrete group is \textit{cyclically generated as a discrete group} if it contains a cyclic automorphism in the sense of \cref{def: cyclicity discrete}.

Let us return to the example at the beginning of the section. Let $Q\le V$ be the subset consisting of eventually periodic sequences, where an \textit{eventually periodic sequence} is a sequence $(e_n)_{n\ge 0}$ where there exist $k,N\ge 0$ such that $e_{k+j+nN}=e_{k+j}$ for every $n\ge 0$ and $0\le j<N$.

\begin{proposition}
\label{proposition: rational subgroup}
    The set $Q$ is a $\tau$-invariant infinite-dimensional dense subspace of $V$ not containing cyclic vectors for the left-shift operator $\tau$. Furthermore, every vector of $Q$ is contained in a finite $\tau$-invariant subspace.
\end{proposition}
\begin{proof}
    Clearly $Q$ is dense in $V$ as every sequence in $V$ may be approximated by eventually periodic sequences. A linear combination of eventually periodic sequences is again an eventually periodic sequence. Similarly the left-shift of an eventually periodic sequence is again an eventually periodic sequence. Thus $Q$ is a $\tau$-invariant dense subspace of $V$. Finally, the $\tau$-orbit of an eventually periodic sequence is finite, so the $\mathbb{F}_p$-subspace which this orbit generates is finite-dimensional over~$\mathbb{F}_p$ and thus finite. 
\end{proof}

Now \cref{proposition: rational subgroup} yields the existence of an infinite self-similar discrete group not admitting cyclic automorphisms in the sense of \cref{def: cyclicity discrete}:

\begin{corollary}
\label{corollary: existence of non-cyclically generated discrete groups}
    There exists an infinite self-similar discrete subgroup $G\le A\le \mathrm{Aut}~T$ such that every $g\in G$ is contained in a finite self-similar subgroup.
\end{corollary}
\begin{proof}
    Just consider $G:=f^{-1}(Q)$ by \cref{lemma: self-similar invariant for A} and \cref{proposition: rational subgroup}.
\end{proof}

Note however that the closure $\overline{G}=A\le \mathrm{Aut}~T$ admits cyclic and hypercyclic automorphisms in the sense of \cref{def: cyclic profinite} by \cref{Theorem: super strongly fractal cyclic}, as $\overline{G}=A$ is super strongly fractal.

Let us see some examples of self-similar discrete groups acting faithfully on the binary rooted tree and admitting cyclic automorphisms.

\subsubsection{First Grigorchuk group} The first Grigorchuk group $\mathcal{G}$ is generated by the rooted automorphism $a$ corresponding to the cycle $\sigma:=(1\,2)\in \mathrm{Sym}(2)$ and the automorphisms $b,c,d$ defined recursively as $\psi(b)=(a,c)$, $\psi(c)=(a,d)$ and $\psi(d)=(1,b)$; see \cite{GrigorchukExample}. Then $\mathcal{G}$ is cyclically generated by any of the $b,c$ and $d$.

\subsubsection{Basilica group} The Basilica group $\mathcal{B}$ is generated by the automorphisms $\psi(a)=(1,b)\sigma$ and $\psi(b)=(1,a)$; see \cite{Basilica}. Then~$\mathcal{B}$ is cyclically generated by both $a$ and $b$. We shall see that every non-trivial automorphism of $\mathcal{B}$ is cyclic.

\subsubsection{Brunner-Sidki-Vieira group} The Brunner-Sidki-Vieira group $H$ is the group generated by $\psi(a)=(1,a)\sigma$ and $\psi(b)=(1,b^{-1})\sigma$; see \cite{BSV1}. Then the generators $a$ and $b$ are not cyclic as each of them generates a distinct proper self-similar subgroup of $H$. However, the element $\psi(ab)=(b^{-1},a)$ is cyclic so $H$ is cyclically generated.\\

The three examples above are instances of automata groups. We devote the remainder of the section to studying cyclicity in automata groups.

\subsection{Automata groups}

We recall that an \textit{invertible finite-state automaton} is a quadruple $(S,X,t,o)$ where:
\begin{enumerate}[\normalfont(i)]
    \item the finite set $S$ is the \textit{set of states};
    \item the finite set $X$ is the \textit{alphabet};
    \item the map $t:X\times S\to S$ is the \textit{transition function};
    \item the map $o:X\times S\to X$ is the \textit{output function} satisfying $o(-,s)\in \mathrm{Sym}(X)$ for all $s\in S$.
\end{enumerate}

Finite-state automata are usually represented as labelled directed graphs, where the set of vertices corresponds to the set of states and edges are given by the transition function and labelled according to the output function; see \cref{fig:enter-label} below. Two states are said to be \textit{path-connected} if there exists a directed path in the automaton from one to the other. A state $s\in S$ is said to be \textit{fully connected} if there is a path from it to any other state in $S$.

If the alphabet $X$ has $m$ symbols we may assume $X=\{1,\dotsc,m\}$. Then the transition and the output functions may be extended to the $m$-adic tree $T$ inductively as
$$t(xv,s)=t(v,t(x,s))\quad\text{and}\quad o(xv,s)=o(x,s)o(v,t(x,s))$$
for $x\in X,s\in S$ and $v\in T$ and both $t(\emptyset,s)=s$ and
$o(\emptyset,s)=\emptyset$. Therefore, for every state $s\in S$ the function $o_s:T\to T$ given by $o_s(v)=o(v,s)$ for $v\in T$ yields an automorphism of $T$, i.e. $o_s\in \mathrm{Aut}~T$. A state $s\in S$ is said to be \textit{finitary} if the corresponding automorphism $o_s$ is finitary.

\begin{figure}[H]
    \centering
    \begin{tikzpicture} [draw=teal!70!black,
    node distance = 4cm, 
    on grid, 
    auto,
    every initial by arrow/.style = {thick}]

\node (c) [state, red] {$c$};  
\node (b) [state, red] at (-2,2) {$b$};
\node (d) [state, red] at (2,2) {$d$};
\node (a) [state] at (-2,-2) {$a$};
\node (1) [state, accepting] at (2,-2) {$1$};

\node (aa) [state, red] at (4,0) {$a$};
\node (bb) [state, red] at (8,0) {$b$};
\node (11) [state, accepting] at (6,-0) {$1$};

% Arrows
\path [-stealth, thick]
     (c) edge[red] node {$1|1$} (d)
     (d) edge[red] node[above] {$1|1$} (b)
     (b) edge[red] node {$1|1$} (c)
     (b) edge node[left] {$0|0$} (a)
     (c) edge node {$0|0$} (a)
     (d) edge node {$0|0$} (1)
     (a) edge node[below] {$0|1,~1|0$} (1)
     (1) edge[loop below]  node {1}()
     (aa) edge[red, bend left=90] node {$1|0$} (bb)
     (bb) edge[red, bend left=90] node {$1|1$} (aa)
    (aa) edge  node {$0|1$} (11)
     (bb) edge  node[above] {$0|0$} (11);
\end{tikzpicture}
    \caption{From left to right, the Grigorchuk and the Basilica automata with the fully connected non-finitary states in red.}
    \label{fig:enter-label}
\end{figure}
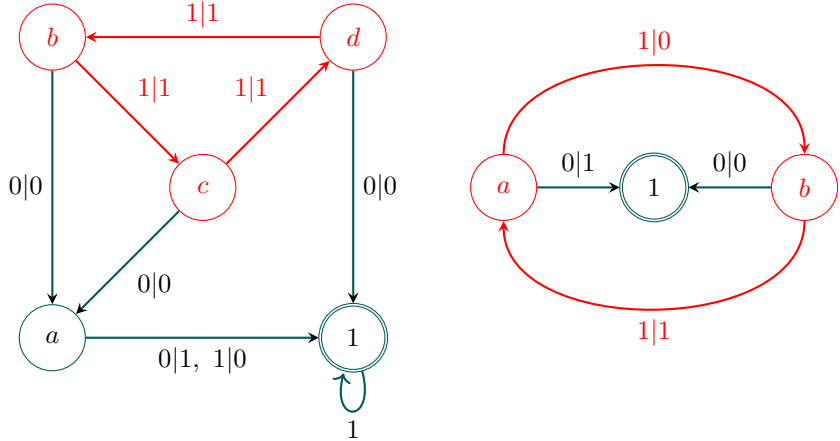

Given a finite-state automaton $(S,X,t,o)$ the subgroup $\langle o_s~:~ s\in S\rangle\le \mathrm{Aut}~T$ is the corresponding \textit{automata group}. We refer the reader to \cite{Automata} for further details on automata groups.

A first observation is that a fully-connected state $s\in S$ yields a cyclic automorphism $o_s\in \mathrm{Aut}~T$.

\begin{proposition}
\label{proposition: cyclic generation in automata groups}
    Let $G\le \mathrm{Aut}~T$ be an automata group given by an automaton $(S,X,t,o)$ containing a fully-connected state $s\in S$. Then $G$ can be cyclically generated by $o_s$.
\end{proposition}
\begin{proof}
    By definition of fully connected, for each state $s_i\in S$ there exists a vertex $v_i\in T$ such that $o_s|_{v_i}=o_{s_i}$. Therefore $\langle o_s\rangle^{\mathrm{SS}}\ge \langle o_{s_i}\mid s_i\in S\rangle=G$ and $o_s$ is cyclic.    
\end{proof}

Let $G$ be a group generated by a finite set $S$. The \textit{length} of $g$ with respect to~$S$ is defined as the minimal length of a word representing $g$ with letters in $S\cup S^{-1}$, and we denote it by $|g|$. If $G\le \mathrm{Aut}~T$ is self-similar then for $g\in G$ and $n\ge 0$ we define $l_n(g)$ as the maximum of the lengths of the sections of $g$ at the $n$th level of~$T$. We say $G$ is \textit{contracting} if there exist constants $C,N\ge 1$ and $\lambda<1$ such that $l_n(g)\le \lambda |g|+C$ for every $n\ge N$ and every $g\in G$. In this case there is a finite set called the \textit{contracting nucleus of} $G$ such that for any $g\in G$ all the sections of~$g$ deep enough in the tree lie in this contracting nucleus.

Now we are in position to state and prove a sufficient condition for cyclicity in contracting automata groups:

\begin{theorem}
\label{Theorem: main result}
    Let $G\le \mathrm{Aut}~T$ be a contracting automata group generated by a set of states $S$ which satisfies the following properties:
    \begin{enumerate}[\normalfont(i)]
        \item the non-finitary states in $S$ are fully connected;
        \item the contracting nucleus of $G$ is contained in $\{o_s~:~s\in S\}$.
    \end{enumerate}
    Then every non-finitary $g\in G$ is cyclic and every finitely generated self-similar proper subgroup of $G$ is finite.
\end{theorem}
\begin{proof}
    If every non-finitary automorphism is cyclic then every finitely generated self-similar proper subgroup is finitary and in particular finite. Thus the second assertion follows from the first one. Let $g\in G$ be a non-finitary automorphism and let us show that $g$ is cyclic. Since $g$ is non-finitary and $G$ is contracting with contracting nucleus contained in the set $\{o_s~:~s\in S\}$, we may find a vertex $v\in T$ such that $g|_{v}=o_s$ where $o_s$ is non-finitary. However, by assumption $s$ is fully connected and thus by \cref{proposition: cyclic generation in automata groups} the automorphism $o_s$ is cyclic. Then $\langle g\rangle^{\text{SS}}\ge \langle o_s\rangle^{\text{SS}}=G$ and $g$ is cyclic as well.
\end{proof}

\cref{Theorem: main result} applies to many well-known automata groups, such as the first Grigorchuk group and the Grigorchuk-Gupta-Sidki groups; see \cite[section 2.3]{GGScontracting}. A consequence of \cref{Theorem: main result} is that torsion-free self-similar automata groups satisfying the assumptions in \cref{Theorem: main result} do not have any non-trivial proper self-similar subgroup. In particular, this is the case for the Basilica group $\mathcal{B}$; see \cite{Basilica}.

\begin{Corollary}
    The Basilica group $\mathcal{B}$ does not contain any non-trivial proper self-similar subgroup.
\end{Corollary}

\section{Dynamical systems and fractality}
\label{section: fractal groups}

The aim of this last section is first to study the measure-preserving dynamical systems arising from fractal profinite groups and second to prove hypercyclicity of Haar-random elements in super strongly fractal profinite groups. For the latter we shall use Birkhoff's ergodic theorem for free semigroup actions.

\subsection{Fractality and measure-preserving dynamical systems} 

Given a profinite group $G$ we may regard it as a probability space $(G,\mu_G)$. Furthermore, if $G\le \mathrm{Aut}~T$ is self-similar we have a well-defined natural action $\mathcal{T}$ of the free monoid~$T$ on $G$. Indeed, for every $v\in T$ we obtain an operator $\mathcal{T}_v:G\to G$ mapping $g\mapsto g|_v$. Thus, for a self-similar profinite group $G$ the tuple $(G,\mu_G,\mathcal{T},T)$ is a measure-preserving dynamical system if this action $\mathcal{T}$ of $T$ on $G$ is measure-preserving. Let us fix $G\le \mathrm{Aut}~T$ a self-similar profinite group for the remainder of the section.

We start by proving a technical lemma for the family of operators $\{\mathcal{T}_v\}_{v\in V}$ in fractal and super strongly fractal groups. Let us fix first some notation. Let $A\subseteq G_n$ be a set of $n$-patterns. We define $C_A$ the \textit{cone over }$A$ as the open subset
$$C_A:=(\varphi_\emptyset^n)^{-1}(A)\subseteq G.$$

\begin{lemma}
\label{lemma: fractal order of fibers of projections Ti}
    Let $n,m\ge 1$ and $v\in T$ a vertex at level $n$. The following holds:
    \begin{enumerate}[\normalfont(i)]
        \item if $G$ is fractal then
        $$\# \varphi_\emptyset^{n+m}(\mathcal{T}_v^{-1}(\mathrm{St}_G(m)))=|G:\mathrm{st}_G(v)|\cdot|\ker \varphi_v^m:\mathrm{St}_G(n+m)|;$$
        \item if $G$ is super strongly fractal then for any $A\subseteq G_n$, $B\subseteq G_m$ sets of $n$- and $m$-patterns respectively we have
         $$\# \varphi_\emptyset^{n+m}(C_A\cap \mathcal{T}_v^{-1}(C_B))=\# A\cdot \# B\cdot |\mathrm{St}_G(n)\cap \ker \varphi_v^m:\mathrm{St}_G(n+m)|.$$
    \end{enumerate}
\end{lemma}
\begin{proof}
    We work modulo $\mathrm{St}_G(n+m)$ without stating it explicitly to simplify notation. Let us prove (i) first. We claim that if $g\in \mathcal{T}_v^{-1}(\mathrm{St}_G(m))$ and $h\in \ker \varphi_{v^g}^m\le \mathrm{st}_G(v^g)$ then $gh$ lies in the same right coset of $\mathrm{st}_G(v)$ as $g$ and $gh\in \mathcal{T}_v^{-1}(\mathrm{St}_G(m))$. Indeed, if $g\in \mathcal{T}_v^{-1}(\mathrm{St}_G(m))$ and $h\in \ker \varphi_{v^g}^m$, then 
    $$(gh)|_v^m:=(g|_v^m)(h|_{v^g}^m)=g|_v^m=1$$
    and therefore $gh\in \mathcal{T}_v^{-1}(\mathrm{St}_G(m))$. Furthermore, since $h\in \ker \varphi_{v^g}^m\le \mathrm{st}_G(v^g)$ we get $\mathrm{st}_G(v)gh=\mathrm{st}_G(v)g$ proving the claim. Now, since if $h\in \mathrm{st}_G(v^g)\setminus \ker \varphi_{v^g}^m $ we clearly have $gh\notin \mathcal{T}_v^{-1}(\mathrm{St}_G(m))$ and for $h_1,h_2\in \ker \varphi_{v^g}^m$ we have $h_1\ne h_2$ if and only if $gh_1\ne gh_2$, our claim implies that for every fixed right coset of $\mathrm{st}_G(v)$ either there is no element or there are exactly $|\ker \varphi_v^m|$ distinct elements in the intersection of $\mathcal{T}_v^{-1}(\mathrm{St}_G(m))$ and this fixed right coset of $\mathrm{st}_G(v)$. Thus it is enough to show that $\mathcal{T}_v^{-1}(\mathrm{St}_G(m))$ contains an element in each right coset of $\mathrm{st}_G(v)$. Let us fix a right coset $\mathrm{st}_G(v)h$. By the fractality of $G$ there exists $\ell\in \mathrm{st}_G(v^h)$ such that $\ell|_{v^h}^m=(h|_v^{m})^{-1}\in G_m$. Thus $g:=h\ell \in \mathcal{T}_v^{-1}(\mathrm{St}_G(m))$ since
    $$g|_v^m=(h\ell)|_v^m=(h|_v^m)(\ell|_{v^h}^m)=1.$$
    Furthermore $\mathrm{st}_G(v)g=\mathrm{st}_G(v)h\ell=\mathrm{st}_G(v)h$ as $\ell\in \mathrm{st}_G(v^h)$ proving (i).

    For part (ii), if $G$ is super strongly fractal we may find an element $g\in G_{n+m}$ such that $g|_\emptyset^n=\sigma$ and $g|_v^m=\tau$ for each pair of elements $\sigma\in G_n$ and $\tau\in G_m$. Indeed, just consider $g:=h\ell$ where $h\in G$ is such that $h|_\emptyset^n=\sigma$ and $\ell\in \mathrm{St}_G(n)$ is such that $\ell|_{v^h}^m=(h|_v^m)^{-1}\tau$. Note that the existence of $h$ and $\ell$ is guaranteed by the super strong fractality of $G$. Therefore, arguing as in (i) (where we take $h\in \mathrm{St}_G(n)\cap \ker \varphi_{v^g}^m$), the number of elements $g\in C_A\cap \mathcal{T}_v^{-1}(C_B)$ such that $g|_\emptyset^n=\sigma$ and $g|_v^m=\tau$ for each distinct pair $\sigma\in A$ and $\tau\in B$ is exactly $|\mathrm{St}_G(n)\cap\ker \varphi_v^m|$, concluding the proof of part (ii).
\end{proof}

Recall that for any $n\ge 1$ and $A\subseteq G_n$ we have 
$$\mu_G(C_A)=\frac{\#A}{|G_n|}$$
by left-translation invariance of $\mu_G$, as $C_A$ is a disjoint union of $\# A$ left-cosets of $\mathrm{St}_G(n)$.

\begin{lemma}
\label{lemma: measure of intersection}
    Let $G$ be super strongly fractal and $m,n\ge 1$. Then for any $A\subseteq G_n$, $B\subseteq G_m$ sets of $n$- and $m$-patterns respectively and any $v\in T$ at level $n$ we have
    $$\mu_G(C_A\cap \mathcal{T}_v^{-1}(C_B))= \mu_G(C_A)\cdot\mu_G(C_B).$$
\end{lemma}
\begin{proof}
Since $G$ is super strongly fractal, we have $\varphi_v^m(\mathrm{St}_G(n))=G_m$. Thus, applying \cref{lemma: fractal order of fibers of projections Ti} we obtain
    \begin{align*}
        \mu_G(C_A\cap \mathcal{T}_v^{-1}(C_B))&=\frac{\# \varphi_\emptyset^{n+m}(C_A\cap \mathcal{T}_v^{-1}(C_B))}{|G_{n+m}|}\\
        &=\frac{\# A}{|G_n|}\cdot \frac{\# B\cdot |\mathrm{St}_G(n)\cap \ker \varphi_v^m:\mathrm{St}_G(n+m)|}{|\mathrm{St}_G(n):\mathrm{St}_G(n+m)|}\\
        &=\frac{\# A}{|G_n|}\cdot \# B\cdot |\mathrm{St}_G(n):\mathrm{St}_G(n)\cap \ker \varphi_v^m|^{-1}\\
    &=\mu_G(C_A)\cdot\frac{\#B}{|\varphi_v^m(\mathrm{St}_G(n))|}\\
    &=\mu_G(C_A)\cdot\frac{\#B}{|G_m|}\\
    &=\mu_G(C_A)\cdot\mu_G(C_B).\qedhere
    \end{align*}
\end{proof}

\begin{proof}[Proof of \cref{Theorem: fractal and ergodic}]
First let us assume $G$ is fractal and let us fix some vertex~$v$ lying at some level $n$. By \cref{corollary: equality of measures} it is enough to check $\mathcal{T}_v$-invariance of $\mu_G$ on the family of level stabilizers $\{\mathrm{St}_G(m)\}_{m\ge 1}$. An application of \cref{lemma: fractal order of fibers of projections Ti} yields 
    \begin{align*}
        \mu_G\big(\mathcal{T}_v^{-1}(\mathrm{St}_G(m))\big)&=\frac{\# \varphi_\emptyset^{n+m}(\mathcal{T}_v^{-1}(\mathrm{St}_G(m)))}{|G_{n+m}|}\\
        &=\frac{|G:\mathrm{st}_G(v)|\cdot |\ker \varphi_v^m:\mathrm{St}_G(n+m)|}{|G_{n+m}|}\\
        &=\frac{|\ker \varphi_v^m:\mathrm{St}_G(n+m)|}{|\mathrm{st}_G(v):\mathrm{St}_G(n+m)|}\\
        &=|\mathrm{st}_G(v):\ker \varphi_v^m|^{-1}\\
        &=|\mathrm{Im}~\varphi_v^m|^{-1},
    \end{align*}
for every $m\ge 1$. Thus by fractality
\begin{align*}
    \mu_G(\mathcal{T}_v^{-1}(\mathrm{St}_G(m)))&=|\mathrm{Im}~\varphi_v^m|^{-1}=|G_m|^{-1}=\mu_G(\mathrm{St}_G(m)).
\end{align*}

Now let us assume $G$ is super strongly fractal. It is enough to check the strong mixing condition on the family of cones $\{C_P\mid P\subseteq G_n \text{ and }n\ge 1\}$, as any measurable subset of $G$ may be approximated (with respect to the Haar measure $\mu_G$) with arbitrary precision by a cone set. Let $A\subseteq G_n$ and $B\subseteq G_m$ for $n,m\ge 1$. For each $k\ge n$, let $A_{k}$ be the unique set of $k$-patterns such that $C_A=C_{A_{k}}$. Then if $v=x_1\dotsb x_k$ is a vertex at the $k$th level of $T$, by \cref{lemma: measure of intersection} we have
$$\mu_G(C_{A_{k}}\cap \mathcal{T}_v^{-1}(C_B))=\mu_G(C_{A_{k}})\cdot\mu_G(C_B).$$
Hence
    \begin{align*}
        \mu_G(C_{A}\cap \mathcal{T}_v^{-1}(C_B))&=\mu_G(C_{A_{k}}\cap \mathcal{T}_v^{-1}(C_B))=\mu_G(C_{A_{k}})\cdot\mu_G(C_B)\\
        &=\mu_G(C_A)\cdot\mu_G(C_B).
    \end{align*}
    Thus the limit
    $$\lim_{k\to \infty}\mu_G(C_{A}\cap \mathcal{T}_{x_1\dotsc x_k}^{-1}(C_B))=\mu_G(C_A)\cdot\mu_G(C_B)$$
    is well-defined and $(G,\mu_G,\mathcal{T},T)$ is strongly mixing.
\end{proof}

\subsection{Hypercyclicity of Haar-random elements}

The first question arising in a measure-preserving dynamical system $(X,\mu,S)$ for a single operator $S$ is how the $S$-iterates of a random point $x\in X$ look like. If $(X,\mu,S)$ is ergodic, then Birkhoff's ergodic theorem yields a nice relation between the time averages and the space averages of the $S$-iterates of any point $x\in X$. Birkhoff's ergodic theorem has been extended from the action of a single operator $S$ to the action of a free group; see \cite{Bufetov, NevoStein}. The case of our interest is the one of free semigroup actions, so we shall use the version of Birkhoff's ergodic theorem in \cite{GrigorchukBirkhoff}; see also \cite{BufetovSemi}.

Let us assume $T$ is the $m$-adic tree, i.e. the free monoid on $m$ generators, and let $\mathbf{p}=\{p_1,\dotsc,p_m\}$ be a probability distribution on the first level of $T$, i.e. on the finite set $\{1,\dotsc,m\}$. We define the \textit{Markov operator} $M$ as the averaging operator
$$M:=\sum_{k=1}^mp_k\mathcal{T}_k.$$
The Markov operator $M$ acts on functions and on measures as
$$(Mf)(x):=\sum_{k=1}^mp_kf(\mathcal{T}_kx)\quad\text{and}\quad M\mu:=\sum_{k=1}^mp_k\mathcal{T}_k\mu,$$
respectively. We say a measure $\mu$ is $M$-\textit{stationary} if $M\mu=\mu$.

Furthermore, for $n\ge 1$ and a function $f$ on $X$ we define the $n$th \textit{Cesàro average} of $f$ as
$$C_nf:=\frac{1}{n}\sum_{|v|\le n}p_v\mathcal{T}_vf,$$
where $p_v$ is the product 
$\prod_{k=1}^{|v|}p_{x_k}$
given by writing $v=x_1\dotsb x_{|v|}\in T$ with each $x_k\in \{1,\dotsc,m\}$.

Now we are in position to state Birkhoff's ergodic theorem for free semigroup actions:

\begin{theorem}[{see \cite[Theorem 1]{GrigorchukBirkhoff}}]
\label{theorem: Birkhoff theorem}
    Let $(X,\mu)$ be a probability space and let $T$ be the free semigroup on $m$ symbols acting via $\mathcal{T}$ on $(X,\mu)$. Assume further that $\mu$ is $M$-stationary. Then for any $1\le p<\infty$ and an arbitrary function $f\in \mathrm{L}^p(X,\mu)$ the Cesàro averages $C_nf$ converge almost continuously to an $M$-invariant function $\widetilde{f}\in \mathrm{L}^p(X,\mu)$ and if $\mu(X)<\infty$, then we have
    $$\int_Xf(x)~d\mu(x)=\int_X\widetilde{f}(x)~d\mu(x).$$
    Furthermore, the almost everywhere defined limit $\widetilde{f}$ is $\mathcal{T}$-invariant.
\end{theorem}

If $\mu(X)<\infty$, the measure-preserving dynamical system $(X,\mu,\mathcal{T}, T)$ is ergodic and $\mathrm{Im}~f\subseteq [0,\infty]\subset \mathbb{R}$, then we further get
$\widetilde{f}$ is constant almost everywhere in \cref{theorem: Birkhoff theorem}. Indeed since $\widetilde{f}$ is $\mathcal{T}$-invariant the subsets 
$$X_\alpha:=\{x\in X~:~ \widetilde{f}(x)\ge \alpha\}$$
are also $\mathcal{T}$-invariant for every $\alpha\ge 0$, and by ergodicity there must exist $\beta\ge 0$ such that $\mu(X_\beta)=1$ and $\mu(X_\alpha)=0$ for every $\beta<\alpha$. Therefore $\widetilde{f}$ is constant almost everywhere and the value $\beta$ is simply given by
$$\beta=\frac{1}{\mu(X)}\int_X f(x)~d\mu(x).$$

For a super strongly fractal profinite group $G$ we consider the strongly mixing measure-preserving dynamical system $(G,\mu_G,\mathcal{T},T)$. In this case $\mu_G(G)=1$. We conclude by proving \cref{Theorem: super strongly fractal cyclic}.

\begin{proof}[Proof of \cref{Theorem: super strongly fractal cyclic}]
    Let us fix $n\ge 1$ and an $n$-pattern $P\in G_n$. It is enough to show that for a Haar-random automorphism $g$ there exists $v\in T$ such that $g|_v^n=P$ almost surely. Indeed, then the $\mathcal{T}$-orbit of $g$ will be dense in $G$ almost surely yielding the hypercyclicity of $g$. Let $A:=C_{P}$ be the cone corresponding to the $n$-pattern~$P$. Then $A$ is measurable in the Borel algebra of $G$ and thus the characteristic function $\chi_A$ is integrable, in other words $\chi_A\in \mathrm{L}^1(G,\mu_G)$. Since $G$ is fractal $\mu_G$ is $\mathcal{T}$-invariant by \cref{Theorem: fractal and ergodic}\textcolor{teal}{(i)} and hence $M$-stationary. Furthermore, since~$G$ is super strongly fractal the measure-preserving dynamical system $(G,\mu_G,\mathcal{T},T)$ is ergodic by \cref{Theorem: fractal and ergodic}\textcolor{teal}{(ii)}. Therefore by \cref{theorem: Birkhoff theorem} the Cesàro averages $C_n \chi_A$ for the uniform distribution $\mathbf{p}=\{1/m,\dotsc,1/m\}$ converge almost everywhere to the constant function $\widetilde{f}\equiv \beta$, where
    $$\beta=\frac{1}{\mu_G(G)}\int_G \chi_A(x)~d\mu_G(x)=\mu_G(A)=|G_n|^{-1}>0.$$
    Hence almost surely $\widetilde{f}(g)=\beta>0$. Since for each $v\in T$ we have $\chi_A(g|_v)=1$ if and only if $g|_v\in A$, we get that almost surely there exist infinitely many vertices $\{v_k\}_{k\ge 1}\subseteq T$ such that $g|_{v_k}^n=P$ as $\lim_{n\to \infty} C_n\chi_A(g)=\beta>0$ almost surely.
\end{proof}

% End of document	

% References

\bibliographystyle{unsrt}

\end{document}